\documentclass{amsart}[12]

\usepackage[a4paper]{geometry}
\usepackage[dvipsnames]{xcolor}

\usepackage{soul,url}

\usepackage{amsmath}
\usepackage{amsthm}
\usepackage{mathrsfs}
\usepackage{amsfonts}
\usepackage{amssymb}
\usepackage{amscd}
\usepackage{array}
\usepackage{amssymb}
\usepackage{comment}
\usepackage[all, cmtip]{xy}
\usepackage{tikz}
\usetikzlibrary{arrows}
\usepackage{tikz-cd}
\usepackage{enumitem}
\usepackage[linewidth=1pt]{mdframed}

\usepackage{hyperref}

\newtheorem{theorem}{Theorem}[subsection]
\newtheorem{lemma}[theorem]{Lemma}

\newtheorem{conj}[theorem]{Conjecture}
\newtheorem{question}[theorem]{Question}
\newtheorem{definition}[theorem]{Definition}
\newtheorem{corollary}[theorem]{Corollary}
\newtheorem{proposition}[theorem]{Proposition}
\newtheorem{remark}[theorem]{Remark}
\newtheorem{example}[theorem]{Example}

\newcommand{\cP}{\mathcal{P}}

\newcommand{\cM}{\mathcal{M}}

\usepackage[backend=biber,
maxbibnames=99,
style=alphabetic,
isbn=false,
url=false,
sorting=nyt]{biblatex}
\newcommand{\ZZ}{\mathbb{Z}}

\newcommand{\HoH}{\mathrm{HH}}

\newcommand{\cC}{\mathcal{C}}

\newcommand{\too}{\longrightarrow}

\newcommand{\Db}{\mathrm{D}^{b}}

\newcommand{\TIP}{\mathrm{TIP}}
\renewcommand{\deg}{\mathbf{deg}}

\newif\ifmoditem
\newcommand{\setupmodenumerate}{%
  \global\moditemfalse
  \let\origmakelabel\makelabel
  \def\moditem##1{\global\moditemtrue\def\mesymbol{##1}\item}%
  \def\makelabel##1{%
    \origmakelabel{##1\ifmoditem\rlap{\mesymbol}\fi\enspace}%
    \global\moditemfalse}%
}

\title{A counterexample to DG version of Han's conjecture}

\author{Yeqin Liu}
\address[Y. Liu]{Department of Mathematics\\
University of Michigan\\
Ann Arbor, MI 48109\\
USA.}
\email{yqnl@umich.edu}

\author{Yu Shen}
 \address[Y. Shen]{Department of Mathematics, Michigan State University\\ East Lansing, MI 48824\\
 USA.}
\email{shenyu5@msu.edu}

\begin{document}

\begin{abstract}
In \cite{Han06}, Han proposed the following conjecture: let $B$ be a finite-dimensional $k$-algebra. If $\HoH_{n}(B)\neq 0$ for only finitely many $n\in \ZZ$, then $B$ is smooth. 
This conjecture can be generalized to the DG setting: let $B$ be a finite-dimensional DG $k$-algebra. If $\HoH_{n}(B)\neq 0$ for only finitely many $n\in \ZZ$, then $B$ is smooth. 
In this note, we show that the DG generalization of Han's conjecture is false.
\end{abstract}

    \maketitle

    \setcounter{tocdepth}{1}
\tableofcontents

\section{Introduction}

In \cite{Han06}, Han proposed the following conjecture:

\begin{conj}\label{Han's conjecture}
Let $B$ be a finite-dimensional algebra over an algebraically
closed field $k$. If the Hochschild homology groups $\operatorname{HH}_n(B)$ are nonzero for only finitely many $n\in \ZZ$, then
$B$ is (homologically) smooth over $k$.
\end{conj}
Originally, Han's conjecture was formulated with the conclusion that $B$ has finite global dimension. By \cite[Corollary 3.19]{RR22}, over a perfect field $k$, a finite-dimensional $k$-algebra has finite global dimension if and only if it is (homologically) smooth. 

For a DG algebra, global dimension is not well-defined. However, smoothness is defined as follows.
Recall that a DG $k$-algebra $B$ is \emph{smooth over $k$} if $B$ is a perfect left $B^e$-module, where
$
B^e = B \otimes_k B^{\mathrm{op}}
$
denotes the enveloping algebra of $B$. 
Using this definition, we generalize Han's conjecture to DG algebras.

\begin{question}\label{question DG}
    Let $B$ be a finite-dimensional DG algebra over an algebraically
    closed field $k$. If the Hochschild homology groups $\operatorname{HH}_n(B)$ are nonzero for only finitely many $n\in \ZZ$, is $B$ smooth over $k$?
\end{question}

Han's conjecture has attracted much interest, and several advances have been made \cite{buchweitz2005finite,Han06,solotar2013hochschild,cibils2022han}, etc. See the survey \cite{da2023survey} for further progress. However, in general the conjecture remains open. The main result of this paper answers its DG generalization (Question~\ref{question DG}) negatively.

\begin{theorem}[Theorem~\ref{theorem-main}]\label{theorem-main-intro}
    Let $k$ be any  field of $\mathrm{char}(k)$=0. There exists a finite-dimensional DG $k$-algebra $B$ such that $\HoH_{*}(B)$ is finite-dimensional, but $B$ is not smooth over $k$. In particular, there are only finitely many $n\in \ZZ$ such that $\HoH_{n}(B)\neq 0$.
\end{theorem}
In \cite{he2023hodge}, the author shows that for a projective nodal or cuspidal curve $C$ (over $\mathbb{C}$), we have
\[
\operatorname{HH}_{n}(C) \neq 0 \quad \text{for all } n \ge -1.
\]
In particular, $\operatorname{HH}_{n}(C)\neq 0$ for arbitrarily large $n$, so these are not examples for Theorem~\ref{theorem-main-intro}. In fact, we do not currently have a geometric example giving a negative answer to Question~\ref{question DG}. It would be very interesting to answer even the following weaker question.

\begin{question}
    Is there a proper singular variety $X$ such that
    \[
        \HoH_{n}(X):=\operatorname{HH}_{n}(\mathcal{P}\mathrm{erf}(X))\neq 0
    \]
    for only finitely many $n\in \ZZ$?
\end{question}

\subsection{Main idea}
For a DG algebra, smoothness and properness are two different finiteness conditions, which can sometimes be “converted’’ into each other. Our idea is to note that the Weyl algebra $A_{1}$ is smooth but not proper, and it has finite Hochschild homology over a field of characteristic $0$. By choosing a smooth categorical compactification, the kernel algebra $B$ becomes a proper but not smooth DG algebra with finite Hochschild homology. 

In general, by \cite{efimov2020categorical} a smooth categorical compactification does not always exist. In this paper, we construct such a compactification for the Weyl algebra $A_{1}$ by showing that it is quasi-isomorphic to a finite-cell DG algebra (Definition~\ref{definition-finite-cell}).
Moreover, under these conditions, it is well known that $B$ can be chosen to be a finite-dimensional DG algebra (see, e.g., \cite{orlov2020finite}).

\subsection{Acknowledgment}
We especially thank Alexander Efimov for many valuable comments on an early draft. Specifically, he points out another geometric construction of smooth categorical compactification of the Weyl algebra $A_{1}$ using $\mathscr{D}$-modules. See Subsection \ref{subsection-geometric-construction}.

We also thank Yunfan He and Alexander Perry for many useful discussions, and we thank Petter Andreas Bergh for helpful comments.

\subsection{Conventions}
Throughout this paper, we work over a  field $k$ of $\operatorname{char}(k)$=0. All DG algebras $A=(A^{\operatorname{gr}},d)$ are DG $k$-algebras, and we use cohomological grading: the differential $d$ has degree $+1$. For a homogeneous element $a\in A$, we write $|a|$ for its degree.

\section{Differential graded algebra}\label{section-DG}

\subsection{DG algebra}

 We assume the reader is familiar with the basic definitions and properties of DG algebras and DG categories; see \cite{keller2006differential,  orlov2020finite,stacks-09JD}  for details.
In this subsection we recall some preliminaries on finite-cell DG algebras.

\begin{definition}\label{definition-finite-cell}
Let $k$ be a field. A \emph{finite-cell} DG algebra over $k$ is a DG algebra $A$ such that
\begin{itemize}
    \item The underlying graded algebra is a finitely generated free graded algebra
$$A^{\mathrm{gr}} \cong k\langle x_1,\cdots,x_n\rangle. $$

\item The differentials satisfy
$
d x_i \in k\langle x_1,\cdots,x_{i-1}\rangle \quad \text{for } i=1,\cdots,n.
$
\end{itemize}
\end{definition}

We also recall the definition of smoothness and properness of DG algebras.

\begin{definition}
A DG \(k\)-algebra \(A\) is \emph{smooth} if \(A\) is perfect as a left DG \(A^{e}\)-module. \(A\) is \emph{proper} if its cohomology
\(H^{*}(A)=\bigoplus_{q} H^{q}(A)\) is finite-dimensional over \(k\).
\end{definition}

It turns out that finite-cell DG algebras have the following nice properties.

\begin{proposition}[{\cite[Remark 2.7]{Efimov2019HomotopyFiniteness}}]\label{proposition-smooth-for-fc}
Finite-cell DG algebras are smooth.
\end{proposition}

\begin{proposition}[{\cite[Theorem 6.2]{Efimov2023SmoothCompactifications}}]\label{fully faithful embedding}
If $A$ is a finite-cell DG algebra, then
$\mathcal{P}\mathrm{erf}(A)$ is Morita equivalent to the DG quotient $\mathcal{E}/\mathcal{S}$, where $\mathcal{E}$ is a pretriangulated proper DG category with a full exceptional collection, and $\mathcal{S}$ is a DG subcategory generated by a single object.
\end{proposition}

We also need the following notions.

\begin{definition}
Let $(A,d_A)$ be a DG algebra, $(M,d_M)$ a (left) DG $A$-module, and let $N\subseteq M$ be a graded $A$-submodule. We define the following DG submodules of $M$:
\begin{itemize}
  \item The \emph{external DG submodule}
  $
    N_{+} := N + d_M(N).
  $
  \item The \emph{internal DG submodule}
  $
    N_{-} := \{\,m\in N \mid d_M(m)\in N\,\}.
  $
\end{itemize}
\end{definition}

\begin{lemma}\label{lemma-M+-}
    Let $A$ be a DG algebra and $M$ a (left) DG $A$-module. If $N\subseteq M$ is a graded $A$-submodule, then the inclusion $N_{-}\hookrightarrow N_{+}$ is a quasi-isomorphism.
\end{lemma}

\begin{proof}
 
    It suffices to show that the quotient complex $N_{+}/N_{-}$ is acyclic.
    Any element of $N_{+}$ can be written as $f+d_M g$ with $f,g\in N$. If
    $[f+d_M g]\in H^{*}(N_{+}/N_{-})$ is a cycle, then
    \[
      d([f+d_M g])=[d_M f]\in N_{-}.
    \]
    Hence $d_M f\in N$, so $f\in N_{-}$ by definition. Therefore
    $[f+d_M g]=[d_M g]=d([g])$, i.e.\ it is a boundary. Thus $H^{*}(N_{+}/N_{-})=0$, and the inclusion $N_{-}\hookrightarrow N_{+}$ is a quasi-isomorphism.
\end{proof}

\section{Order on monomials}\label{section-order}

In this section, we recall the notion of a linear order on monomials in a free algebra and use it to show that certain two-sided DG ideals in finite-cell DG algebras are acyclic.

\begin{definition}\label{definition-order}
    Let $A = k\langle x_{1}, \dots, x_{n}\rangle$ be a free algebra, and let $<$ be a linear order on all nonzero monomials in $A$ (up to scaling). 

    For any nonzero polynomial $f \in A$, let $\mathrm{TIP}(f)$ denote the largest monomial occurring in $f$ with respect to $<$. We say the order $<$ is preserved by left multiplication, if for every nonzero $u, v\in A$, we have $\TIP(v)\leq \TIP(uv)$.
\end{definition}

\begin{definition}\label{definition-unique-order}
Let $A = k\langle x_{1}, \dots, x_{n}\rangle$ be a graded free algebra, and let $<$ be a linear order on all nonzero monomials in $A$. 

    We say that a pair of distinct nonzero (grading) homogeneous elements $(r,\delta) \in A \times A$ satisfies the \emph{unique order property} with respect to $<$ if the following conditions hold:
    \begin{enumerate}[label=(\roman*)]    
        \item There exists a set $\mathcal{V}$ of monomials in $A$ whose images form a $k$-basis of $A/(r, \delta)$; equivalently,
\[
    A = \operatorname{Span}_{k}(\mathcal{V}) \oplus (r,\delta)
\]
as a $k$-vector space, where $(r,\delta)$ denotes the two-sided ideal generated by $r$ and $\delta$.

        \item For any $v_{1}, v_{2} \in \mathcal{V}$, any $\beta_{1}, \beta_{2} \in \{r,\delta \}$, and any nonzero $f_{1}, f_{2} \in A$, we have
\[
    \mathrm{TIP}(f_{1}\beta_{1} v_{1}) \neq \mathrm{TIP}(f_{2}\beta_{2} v_{2})
    \quad\text{whenever}\quad
    (\beta_{1}, v_{1}) \neq (\beta_{2}, v_{2}).
\]

        \item For every $f \in (r,\delta)$ and every $v \in \mathcal{V}$, we have $\mathrm{TIP}(f) \neq v$.
    \end{enumerate}
\end{definition}

\begin{lemma}\label{lemma-unique-decomposition}
Let $(A, <)$ be a free algebra with a linear order $<$ on monomials. If a pair $(r, \delta)$ satisfies the unique order property, then the two-sided ideal $(r, \delta)$ can be decomposed into left ideals
\[
(r, \delta)=\bigoplus_{v_{\ell}\in\mathcal{V}}\left(Ar v_{\ell} \oplus A\delta  v_{\ell} \right).
\]
\end{lemma}

\begin{proof}
    Clearly, we have 
    \[
    (r, \delta)=\sum_{v_{\ell}\in\mathcal{V}}\left(Ar v_{\ell} \oplus A\delta  v_{\ell} \right).
    \]
    To prove it is a direct sum, we need to show that every element $f\in(r, \delta)$ can be uniquely written as 
    \[
    f=\sum a_{i} r v_{i} + \sum b_{j}\delta v_{j}, \quad v_{i}, v_{j}\in \mathcal{V},\ a_{i}, b_{j}\in A.
    \]
    It suffices to assume $f=0$. Suppose
    \begin{equation}\label{equation-0} 
        0=\sum_{i} a_{i} \beta_{i} v_{i}, \quad v_{i}\in \mathcal{V},\ \beta_{i}\in \{\delta, r \},\ 0\neq a_{i}\in A.
    \end{equation}
    Here $(\beta_{i}, v_{i})\neq(\beta_{j}, v_{j})$ for different $i$ and $j$. By (ii) of the unique order property (Definition~\ref{definition-unique-order}), we may assume $\TIP(a_{i} \beta_{i} v_{i})>\TIP(a_{i+1} \beta_{i+1} v_{i+1})$ for all $i$. Hence 
    \[
    \TIP\Big(\sum_{i\geq 2} a_{i} \beta_{i} v_{i}\Big)=\TIP(a_{2} \beta_{2} v_{2})<\TIP(-a_{1} \beta_{1} v_{1}).
    \]
    However, by (\ref{equation-0}), they should be equal, a contradiction. Hence all $a_{i}=0$.
\end{proof}

\subsection{A technical lemma}

In this section, we prove the key technical lemma of this paper.

\begin{lemma}\label{lemma-ideal-acyclic}
    Let $(A, d, <)$ be a finite-cell DG algebra with a linear order $<$ on monomials preserved by left multiplication (Definition \ref{definition-order}). Assume $\TIP(df)<\TIP(f)$ for every $f\in A$.

    If a pair of distinct nonzero homogeneous elements $(r, \delta)$ satisfies the unique order property (Definition \ref{definition-unique-order}) and $dr=\delta$,
    then the two-sided DG ideal $(r, \delta)$ is acyclic.
\end{lemma}

\begin{proof}
 Write $A=(r, \delta)\oplus V$ for a $k$-vector space $V$ such that
    $\mathcal{V}=\{f_{i}: i\in I\}$ is a monomial basis of $V$. 
    
   By Lemma \ref{lemma-unique-decomposition}, we have
    $(r, \delta)=\bigoplus_{v_{\ell}\in\mathcal{V}}(Ar v_{\ell} \oplus A\delta v_{\ell} )$ as left DG ideals.
    Let $M=\bigoplus_{v_{\ell}\in\mathcal{V}} A r v_{\ell}$ and $N=\bigoplus_{v_{\ell}\in\mathcal{V}} A \delta v_{\ell}$. Then $M$ is not a left DG ideal. But we have
\[
M_{+}=M+d(M)=\bigoplus_{v_{\ell}\in\mathcal{V}}\bigl(A \delta  v_{\ell}\oplus A r  v_{\ell}\bigr)=(\delta ,r )
\]
as left DG ideals. By Lemma \ref{lemma-M+-}, we will show that $(r, \delta)=M_{+}$ is acyclic by showing $M_{-}=0$.

Suppose $M_{-} \neq 0$. Since $M_{-} \subseteq M$, there is a homogeneous element
$
0\neq \sum_{\ell} g_{\ell} r v_{\ell} \in M_{-}.
$
Here $g_{\ell} \in A$, and $g_{\ell}, v_{\ell} \neq 0$ for all $\ell$. Moreover, we may assume that the $v_{\ell}$ are pairwise distinct.
By definition of $M_{-}$, we have 
    $$d(\sum_{\ell} g_{\ell}rv_{\ell})=\sum_{\ell}(dg_{\ell})rv_{\ell}+\sum_{\ell}(-1)^{|g_{\ell}|} g_{\ell}\delta v_{\ell}+\sum_{\ell}(-1)^{|g_{\ell}|+|r|} g_{\ell}r(dv_{\ell})\in M.$$
    Since $\sum_{\ell} (dg_{\ell}) r v_{\ell}\in M$, we have $\sum_{\ell}(-1)^{|g_{\ell}|} g_{\ell}\delta v_{\ell}+\sum_{\ell}(-1)^{|g_{\ell}|+|r|} g_{\ell}r(dv_{\ell})\in M$. For every $\ell$, we may write 
    \begin{equation}\label{equation-dv}
dv_{\ell}=h^{0}_{\ell}+h_{\ell}^{1}+\sum_{s} g^{s}_{\ell}\delta u^{s}_{\ell}, \quad h_{\ell}^{0}\in V, h^{1}_{\ell}\in M, g^{s}_{\ell} \in A, u^{s}_{\ell}\in \mathcal{V}. 
    \end{equation}
    Here, for every fixed $\ell$, $u_{\ell}^{s}$ are distinct for different $s$. By definition of $M$, and noting that $M$ is a left ideal, we have 
    $$\sum_{\ell} (-1)^{|g_{\ell}|+|r|} g_{\ell}rh^{0}_{\ell}, \sum_{\ell} (-1)^{|g_{\ell}|+|r|} g_{\ell} r h^{1}_{\ell}\in M.$$ 
    Hence we must have \begin{equation}\label{equation-relation}
    \sum_{\ell}(-1)^{|g_\ell|}
    g_{\ell}\delta v_{\ell}+\sum_{\ell}\sum_{s}(-1)^{|g_{\ell}|+|r|}g_{\ell}r g^{s}_{\ell}\delta  u^{s}_{\ell}
    \in M. 
    \end{equation}
For every $\ell$, by (ii) of Definition \ref{definition-unique-order}, we have $\TIP(\sum_{s} g_{\ell}^{s}\delta u_{\ell}^s)\neq \TIP(h^1_{\ell})$. Hence 
$$\TIP(\sum_{s} g_{\ell}^{s}\delta u_{\ell}^s)\leq \TIP(h^1_{\ell}+\sum_{s} g_{\ell}^{s}\delta u_{\ell}^s).$$
Since $h_{\ell}^{1}+\sum_{s}g_{\ell}^{s}\delta u_{\ell}^s\in (r,\delta)$ and $v^0_{\ell}\in \mathcal{V}$, by (iii) of Definition \ref{definition-unique-order} we have $\TIP(v^0_{\ell})\neq \TIP(h^1_{\ell}+\sum_{s} g_{\ell}^{s}\delta u_{\ell}^s)$. Hence 
\begin{equation}\label{equation-less-than-d}
\TIP(\sum_{s} g_{\ell}^{s}\delta u_{\ell}^s)\leq \TIP(h^1_{\ell}+\sum_{s} g_{\ell}^{s}\delta u_{\ell}^s)\leq \TIP(h^0_{\ell}+h^1_{\ell}+\sum_{s} g_{\ell}^{s}\delta u_{\ell}^s)=\TIP(dv_{\ell}).
\end{equation}
Hence for every $\ell, s$, we have
\begin{equation}\label{equation-compare}
  u_{\ell}^s\leq \TIP(g_{\ell}^s \delta u_{\ell}^s)\leq \TIP(\sum_{t} g_{\ell}^t \delta u_{\ell}^t)\leq \TIP(dv_{\ell})<v_{\ell}.
\end{equation}
    The first inequality is because $<$ is preserved by left multiplication. The second inequality is by (ii) of Definition \ref{definition-unique-order}. The thrid inequality is (\ref{equation-less-than-d}). The last inequality is by assumption $\TIP(df)<\TIP(f)$ for every $f\in A$.
    Since (\ref{equation-relation}) is in $N$ by definition and $M \cap N=0$, we have (\ref{equation-relation})$=0$. We may rewrite (\ref{equation-relation}) in the form
    \begin{equation}\label{equation-contradiction}
    0=\sum_{\ell}
   (-1)^{|g_{\ell}|} g_{\ell}\delta v_{\ell}+\sum_{\ell}\sum_{s}(-1)^{|g_{\ell}|+|r|} g_{\ell}r g^{s}_{\ell}\delta  u^{s}_{\ell}=\sum_{j} h_{j} \delta w_{j},\quad h_{j}\in A, w_{j}\in \mathcal{V}.
    \end{equation}
  Here, $w_{j}\in \mathcal{V}$ are distinct. Assume $v_{1}=\max\{ v_{\ell} \}$. By (\ref{equation-compare}), $u_{\ell}^s<v_{1}$ for all $\ell, s$. So the $(-1)^{|g_{1}|}g_{1}\delta v_{1}$ term is not affected. We may assume $w_{1}=\max\{ w_{j} \}=v_{1}$, and then $h_{1}=(-1)^{|g_{1}|}g_{1}\neq 0$.
However, by applying Lemma~\ref{lemma-unique-decomposition} to (\ref{equation-contradiction}), we obtain $h_{j}=0$ for all $j$, a contradiction.

So we have $M_{-}=0$.
Since $M_{+}$ is quasi-isomorphic to $M_{-}$ by Lemma~\ref{lemma-M+-}, we have $H^{\ast}(M_{+})=0$, so $M_{+}$ is acyclic.
\end{proof}

\subsection{Examples}
For better understanding of the notions introduced in this section, we now present some explicit examples using the right lexicographic order.

\begin{definition}[Right lexicographic order]\label{definition-right-lexi}
    Let $(S,\prec)$ be a linearly ordered set. The \emph{right lexicographic order} on $S^{\mathbb{N}}$ is defined as follows.
    We write
    \[
        (\dots, s_{2}, s_{1}) \prec (\dots, s'_{2}, s'_{1})
    \]
    if there exists an index $i \ge 1$ such that
    \[
        s_j = s'_j \quad \text{for all } j < i,
        \qquad\text{and}\qquad
        s_i \prec s'_i.
    \]
\end{definition}

The following is an explicit example of right lexicogrphic order.
\begin{example}\label{example-naive}

Let $A=(k\langle x_{1}, \cdots, x_{n} \rangle, \prec )$ be a free algebra equipped with the right lexicographic order induced by
\[
1\prec x_{1}\prec x_{2}\prec \cdots \prec x_{n}.
\]
For example, let
\[
u = x_{1}^{2}x_{2}x_{3}^3,\quad
v = x_{3}x_{1}x_{2}x_{3}^2x_{1}x_{3},\quad
w = x_{3}x_{1}^{5}x_{2}^3 x_{3},
\]
and $f = 4u + 7v + 2w$.
The rightmost letters of $u, v, w$ are all $x_{3}$, so we compare the second rightmost letters. The second rightmost letter of $u$ is $x_{3}$, that of $v$ is $x_{1}$, and that of $w$ is $x_{2}$. Hence we have
\[
v \prec w \prec u, \quad \TIP(f)=4u.
\]

The pair $(x_{i}, x_{j})$ satisfies the unique order property (Definition~\ref{definition-unique-order}).
Concretely, let 
$
    \mathcal{V}=\{x_{i_{1}}x_{i_{2}}\cdots x_{i_{r}} \mid i_{k}\neq i,j \text{ for all } k\}.
$
Then $\mathcal{V}$ is a monomial basis of the $k$-vector space 
$
V = k\langle x_{1}, \cdots, x_{i-1}, x_{i+1}, \cdots, x_{j-1}, x_{j+1}, \cdots, x_{n} \rangle,
$
and we have $A = V \oplus (x_{i}, x_{j})$. Moreover,
$
    (x_{i}, x_{j})=\bigoplus_{v_{\ell}\in\mathcal{V}}\bigl(Ax_{i}v_{\ell} \oplus Ax_{j}v_{\ell}\bigr)
$
as a left ideal. Hence conditions (i) and (ii) in Definition~\ref{definition-unique-order} hold. 
For every $f\in (x_{i}, x_{j})$ and $v\in V$, $\TIP(f)$ is a monomial in $(x_{i}, x_{j})$ and $v\notin (x_{i}, x_{j})$. Hence condition (iii) also holds.

\end{example}

\begin{corollary}
    Let $A=(k\langle x_{1}, \cdots, x_{n}\rangle, d)$ be a finite cell DG algebra. If $dx_{j}=x_{i}$ for a pair $i<j$, then the 2-sided ideal $(x_{i}, x_{j})$ is acyclic.
\end{corollary}

\begin{proof}

Note that the order $\prec$ defined in Example~\ref{example-naive} is preserved by left multiplication.
By Example~\ref{example-naive}, the right lexicographic order $\prec$ on $A$ makes $(x_{i}, x_{j})$ satisfy the unique order property.
By the definition of a finite-cell DG algebra, we have $\TIP(df)\prec\TIP(f)$ for all $f\in A$. Hence, by Lemma~\ref{lemma-ideal-acyclic}, $(x_{i}, x_{j})$ is acyclic.
\end{proof}

\begin{remark}
If $(r, \delta)$ is an arbitrary pair such that $dr=\delta$ and $\delta\neq 0$, then the two-sided ideal $(r, \delta)$ may not be acyclic in general. For instance, let $A=k\langle x_{1}, x_{2} \rangle$ where $dx_{1}=0$ and $dx_{2}=x_{1}^{2}$.
Then one may explicitly compute that $H^{*}(A)$ is infinite dimensional. If $(x_{2}, x_{1}^{2})$ were acyclic, then $A$ would be quasi-isomorphic to
\[
A/(x_{2}, x_{1}^{2})\cong k\langle x_{1} \rangle/(x_{1}^{2}),
\]
which is finite dimensional. Hence $(x_{2}, x_{1}^{2})$ is not acyclic. The reason is that condition (ii) in Definition~\ref{definition-unique-order} is not satisfied. 
\end{remark}

\section{Main result}\label{main result}

Recall that we introduced the right lexicographic order in Definition~\ref{definition-right-lexi}. However, for the purpose of proving our main theorem, this order is not suitable (see Remark~\ref{remark-degree-order-reason}). We therefore introduce a different order on nonzero monomials in a free algebra.

\begin{definition}[Degree order]\label{definition-degree-order}
    Let $A=k\langle x_{1}, \cdots, x_{n}\rangle$ be a free $k$-algebra, and let $u\in A$ be a monomial. The \emph{degree vector} of $u$ is defined as 
    \[
        \deg(u)=(m_{1}, \cdots, m_{n})\in \ZZ_{\geq 0}^{n},
    \]
    where $m_{i}$ is the number of occurrences of $x_{i}$ in $u$. 

    We put the natural order $\prec$ on $\ZZ_{\geq 0}$ and the right lexicographic order $\prec$ on the space of degree vectors $\ZZ^{n}_{\geq 0}$. For monomials $u,v\in A$, we define $v<u$ if $\deg(v)\prec \deg(u)$, or $\deg(v)=\deg(u)$ and $v \prec u$ in the right lexicographic order (Definition~\ref{definition-right-lexi}) induced by
    \[
        1\prec x_{1}\prec x_{2}\prec \cdots \prec x_{n}.
    \]
\end{definition}

The following is an explicit example of the degree order.

\begin{example}
Let $A=k\langle x_{1}, x_{2}, x_{3}\rangle$ be a free algebra. Let
\[
u = x_{1}^{2}x_{2}x_{3}^3,\quad
v = x_{3}x_{1}x_{2}x_{3}^2x_{1}x_{3},\quad
w = x_{3}x_{1}^{5}x_{2}^3 x_{3},\quad
f = 4u+7v+2w.
\]
Then we have 
\[
\deg(u)=(2,1,3),\quad \deg(v)=(2,1,3),\quad \deg(w)=(5,3,2).  
\]
So $\deg(w)\prec\deg(u)=\deg(v).$ Hence $w<u$ and $w<v$. The second letter from the right in $u$ is $x_{3}$, and that in $v$ is $x_{1}$. Hence, by Definition~\ref{definition-right-lexi}, $v\prec u$. By Definition~\ref{definition-degree-order}, we have 
$
w<v<u,\quad  \TIP(f)=4u.
$
\end{example}

It turns out that the degree order enjoys exactly the properties we need.
\begin{lemma}\label{lemma-TIP}
    Let $A=(k\langle x_{1},\ldots,x_{n}\rangle, <)$ be a finite-cell DG algebra equipped with the degree order (Definition \ref{definition-degree-order}). For every polynomial $f\in A$, we have $\mathrm{TIP}(df)<\TIP(f)$.
\end{lemma}

\begin{proof}
We may assume that $f$ is a monomial.
By Definition~\ref{definition-finite-cell}, for all $i$ we have $dx_{i} \in k\langle x_{1},\ldots,x_{i-1}\rangle$. Hence, for every monomial $g$ occurring in $df$, we have $\deg(g)\prec\deg(f)$.
\end{proof}

\begin{lemma}\label{lemma-unique-order}
    Let $A = (k\langle x_{1}, x_{2}, x_{3} \rangle, d, <)$ be a finite-cell DG algebra where $|x_{1}| = |x_{2}| = 0$ and $|x_{3}| = -1$, and 
    \[
        d x_{1} = d x_{2} = 0
        \quad\text{and}\quad
        d x_{3} = \delta := x_{1} x_{2} - x_{2} x_{1} - 1.
    \]
    Then the pair $(x_{3}, \delta)$ satisfies the unique order property (Definition~\ref{definition-unique-order}).
    Here, $<$ is the degree order in Definition~\ref{definition-degree-order}.
\end{lemma}

\begin{proof}

To show that the pair $(x_{3}, \delta)$ satisfies the first requirement (i) of the unique order property, we just let $\mathcal{V}$ be the set
$
\mathcal{V} = \{ x_{2}^{i} x_{1}^{j} : i,j \geq 0 \}.
$. Then $A=\operatorname{Span}_{k}(\mathcal{V}) \oplus (x_{3},\delta)$.

To prove that it satisfies (ii) of the unique order property, we need to show that for any $v_{1}, v_{2} \in \mathcal{V}$, any $\beta_{1}, \beta_{2} \in \{x_{3},\delta\}$, and  nonzero $f_{1}, f_{2} \in A$, we have
$
    \mathrm{TIP}(f_{1}\beta_{1} v_{1}) \neq \mathrm{TIP}(f_{2}\beta_{2} v_{2})
    \quad\text{whenever}\quad
    (\beta_{1}, v_{1}) \neq (\beta_{2}, v_{2}).
$  We first deal with the case $v_{1} \neq v_{2}$. By the definition of $\mathcal{V}$, we have $A v_{1} \cap A v_{2} = 0$. So the claim is true. 
When $v_{1} = v_{2} = x_{2}^{i}x_{1}^{j}$, we may assume $\beta_{1} = x_{3}$ and $\beta_{2} = \delta$. Since
$$
f_{2} \delta x_{2}^{i} x_{1}^{j} = f_{2} (x_{1} x_{2} - x_{2} x_{1} - 1) x_{2}^{i} x_{1}^{j},
$$
we have 
\[
\TIP(f_{2} \delta v_{2}) = \TIP(f_{2} x_{1} x_{2} v_{2}) = \TIP(f_{2} x_{1} x_{2} v_{1}).
\]
Since $A(x_{1} x_{2} v_{1}) \cap A(x_{3} v_{1}) = 0$, we must have 
\[
\TIP(f_{2} x_{1} x_{2} v_{1}) \neq \TIP(f_{1} x_{3} v_{1}).
\]
Hence (ii) is proved.

To prove it satisfies (iii) of the unique order property, we need to show that   for every $f \in (x_{3},\delta)$ and every $v \in \mathcal{V}$, we have $\mathrm{TIP}(f) \neq v$. Let $f \in (x_{3}, \delta)$. First, if $f \not\in (\delta)$, then $\TIP(f) \geq x_{3} > v$ for every $v \in \mathcal{V}$. Hence we may assume $f \in (\delta)$. Then there exist monomials $a_{\ell}, b_{\ell}$ such that 
\[
f = \sum a_{\ell} (x_{1} x_{2} - x_{2} x_{1} - 1) b_{\ell}.
\]
So we have
\[
\TIP(f) = a_{k} x_{1} x_{2} b_{k} \quad \text{for some } k.
\]
It is not of the form $x_{2}^{i} x_{1}^{j}$, so the claim is proved.
\end{proof}

\begin{corollary}\label{corollary-acyclic-ideal}
    Let $A=k\langle x_{1},x_{2}, x_{3} \rangle$ be the finite-cell DG algebra where $|x_{1}|=|x_{2}|=0, |r|=-1$, and the differentials are
    $$
    dx_{1}=dx_{2} =0, dx_{3}=\delta=x_{1}x_{2} -x_{2}x_{1}-1.
    $$
    Then the 2-sided ideal $(x_{3}, \delta)$ is an acyclic DG ideal.
\end{corollary}

\begin{proof}
By Definition \ref{definition-degree-order}, the degree order $<$ is preserved by left multiplication.
    By Lemma \ref{lemma-TIP} and Lemma \ref{lemma-unique-order}, the conditions of Lemma \ref{lemma-ideal-acyclic} hold.
\end{proof}

Note that the right lexicographic order (Definition \ref{definition-right-lexi}) is not suitable.
\begin{remark}\label{remark-degree-order-reason}
    Let $A$ be the finite-cell DG algebra in Corollary~\ref{corollary-acyclic-ideal}. The pair $(x_{3},\delta)$ does not satisfy the unique order property (Definition~\ref{definition-unique-order}) with respect to the right lexicographic order $\prec$ in Example \ref{example-naive} and the set $\mathcal{V} = \{ x_{2}^{i} x_{1}^{j} : i,j \geq 0\}$. Indeed, in this case we have 
    \[
        \TIP(x_{3} \cdot \delta \cdot 1)
        = \TIP(x_{3} x_{1} x_{2} - x_{3} x_{2} x_{1} - x_{3})
        = x_{3}
        = \TIP(1 \cdot x_{3} \cdot 1),
    \]
    so condition~(ii) of the unique order property is not satisfied.
\end{remark}

Now we prove the main theorem.

\begin{theorem}\label{theorem-main}
    Let $k$ be any  field of $\mathrm{char}(k)=0$. There is a finite dimensional DG $k$-algebra $B$ such that $\HoH_{*}(B)$ is finite dimensional, but $B$ is not smooth over $k$. In particular, there are only finitely many $n\in \ZZ$ such that $\HoH_{n}(B)\neq 0$.
\end{theorem}

\begin{proof}
Let $A$ be the finite-cell DG 
$k$-algebra in Corollary~\ref{corollary-acyclic-ideal}. Since $(x_{3}, x_{1} x_{2}-x_{2}x_{1}-1)$ is an acyclic DG ideal, $A$ is quasi-isomorphic to the DG algebra
\[
A/(x_{3}, x_{1} x_{2}-x_{2}x_{1}-1) \cong A_{1} := \bigl(k\langle x_{1},x_{2} \rangle/(x_{1}x_{2}-x_{2}x_{1}-1), d_{A_{1}}\bigr),
\]
where $|x_{1}| = |x_{2}| = 0$ and $d_{A_{1}}(x_{1}) = d_{A_{1}}(x_{2}) = 0$.
Note that $A_{1}$ is the classical Weyl algebra.
    Hence by \cite[Theorem 2.1]{farinati2003hochschild}, we have 
    \begin{equation}\label{equation-HH-Weyl}
    \HoH_{i}(\cP\mathrm{erf}
    (A))\cong
\operatorname{HH}_i(A_1) =
\begin{cases}
k & \text{if } i = 2,\\
0 & \text{if } i \neq 2.
\end{cases}
\end{equation}

 Since $A$ is a finite-cell DG algebra, by Proposition~\ref{fully faithful embedding} there is a proper pretriangulated DG category $\cC$ that admits a full exceptional collection 
\begin{equation}\label{equation-SOD}
    \cC = \langle E_{1}, \dots, E_{m} \rangle,
\end{equation}
such that $\mathcal{P}\mathrm{erf}(A)$ is a quotient of $\cC$ by a full idempotent complete pretriangulated DG subcategory generated by a single object $S$:
$
    \mathcal{P}\mathrm{erf}(A) \cong \cC / \langle S \rangle.
$
By \cite[Corollary~2.20]{orlov2020finite}, the subcategory $\langle S \rangle$ is quasi-equivalent to $\mathcal{P}\mathrm{erf}(B)$ for a finite-dimensional DG algebra $B$.

We claim that $B$ is not smooth. Assume that $B$ is a smooth DG algebra. Since $B$ is finite-dimensional, it is automatically proper. By \cite[Corollary~2.20]{orlov2020finite}, the inclusion
\[
\mathcal{P}\mathrm{erf}(B) = \langle S \rangle \subset \cC
\]
is then admissible. Hence
$
\cC / \langle S \rangle = \mathcal{P}\mathrm{erf}(A)
$
is an admissible subcategory of $\cC$. In particular, $\mathcal{P}\mathrm{erf}(A)$ is proper, and thus $A$ is a proper DG algebra. Hence $H^{*}(A)$ is finite-dimensional. However, $A$ is quasi-isomorphic to $A_{1}$, so
\[
H^{*}(A) = H^{*}(A_{1}) = H^{0}(A_{1}) = A_{1},
\]
which is infinite-dimensional, a contradiction. Hence $B$ is not smooth.

    By \cite[Theorem 3.1]{keller1998invariance}, $\HoH_{*}(-)$ has a long exact sequence 
    with respect to localization: 
    $$\HoH_{*}(\mathcal{P}\mathrm{erf}(B)) \to \HoH_{*}(\cC) \to \HoH_{*}(\mathcal{P}\mathrm{erf}(A)) \to \HoH_{*-1}(\mathcal{P}\mathrm{erf}(B)).$$
  By (\ref{equation-SOD}), we have $$\operatorname{HH}_i(\cC) =
\begin{cases}
k^{\oplus m} & \text{if } i = 0,\\
0 & \text{if } i \neq 0.
\end{cases}$$
So by (\ref{equation-HH-Weyl}), we have

$$\operatorname{HH}_i(\mathcal{P}\mathrm{erf}(B))=
\begin{cases}
k^{\oplus m} & \text{if } i = 0,\\
k & \text{if } i =1,\\
0 & \text{if } i\neq 0,1.
\end{cases}$$
In particular,  $\HoH_{*}(B)\cong \HoH_{*}(\mathcal{P}\mathrm{erf}(B))$ is finite dimensional. 
\end{proof}

\subsection{Another smooth categorical compactification}\label{subsection-geometric-construction}

The construction in this subsection is due to Alexander Efimov. 
In Theorem \ref{theorem-main}, we show that the Weyl algebra $A_{1}$ admits a smooth categorical compactification by proving that $A_{1}$ is a finite-cell DG algebra. Efimov provides another smooth compactification using $\mathscr{D}$-modules. We illustrate his idea as follows.

Let $F\Db(\mathscr{D}_{\mathbb{P}^{1}})$ be the bounded filtered derived category of  coherent $\mathscr{D}_{\mathbb{P}^{1}}$-modules equipped with a good filtration. An object $(\mathcal{M}^{\bullet}, F_{\bullet})\in F\Db(\mathscr{D}_{\mathbb{P}^{1}})$ consists of $\cM^{\bullet}\in \Db(\mathscr{D}_{\mathbb{P}^{1}})$ and a good filtration $F_{\bullet}$ on $\cM^{\bullet}$.
Let $\cC \subset F\Db(\mathscr{D}_{\mathbb{P}^{1}})$ be the full triangulated subcategory generated by objects $(\mathcal{M}^{\bullet}, F_{\bullet})$ such that the complex $F_{j}\mathcal{M}^{\bullet}$ is acyclic for $j\gg 0$.  It can be shown that the quotient category
$
\mathcal{D} := F\Db(\mathscr{D}_{\mathbb{P}^{1}})/\cC
$
is equivalent to the derived category of a noncommutative $\mathbb{P}^{1}$-bundle over $\mathbb{P}^{1}$ (see \cite{ben2008perverse} for a description of this noncommutative $\mathbb{P}^{1}$-bundle).
In particular, $\mathcal{D}$ is smooth and proper, and it admits a full exceptional collection of length $4$.

We have the forgetful functor
$$
G \colon F\Db(\mathscr{D}_{\mathbb{P}^{1}}) \to \Db(\mathscr{D}_{\mathbb{P}^{1}}), 
\quad (\mathcal{M}^{\bullet}, F_{\bullet}) \mapsto \mathcal{M}^{\bullet}.
$$
For every $(\mathcal{M}^{\bullet}, F_{\bullet})\in \cC\subset F\Db(\mathscr{D}_{\mathbb{P}^{1}})$, it can be checked that $G((\mathcal{M}^{\bullet}, F_{\bullet}))=\mathcal{M}^{\bullet}$ is acyclic. So $G$ induces a functor
$
G' \colon F\Db(\mathscr{D}_{\mathbb{P}^{1}})/\cC=\mathcal{D} \too \Db(\mathscr{D}_{\mathbb{P}^{1}}).
$
Note that we have an open embedding $i \colon \mathbb{A}^{1} \hookrightarrow \mathbb{P}^{1}$, which induces a functor
$
i^{*} \colon \Db(\mathscr{D}_{\mathbb{P}^{1}}) \to \Db(\mathscr{D}_{\mathbb{A}^{1}}) = \mathrm{Perf}(A_{1}).
$
Thus we have the functor
\[
T = i^{*} \circ G' \colon \mathcal{D} \to \mathrm{Perf}(A_{1}).
\]
Clearly $T$ is a quotient functor. Moreover it can be shown that the kernel of $T$ is generated by a single object. Therefore, $\mathcal{D}$ is a smooth categorical compactification of $\mathrm{Perf}(A_{1})$.

\printbibliography

\end{document}